\documentclass[11pt]{amsart}
\usepackage[margin=1in]{geometry}
\usepackage{xcolor}
\usepackage[english]{babel}

 \newcommand{\irnn}{\int_{\Omega_{n }(t)}}

\def\XXint#1#2#3{{\setbox0=\hbox{$#1{#2#3}{\int}$ }
		\vcenter{\hbox{$#2#3$ }}\kern-.6\wd0}}

 \numberwithin{equation}{section}
 \newcommand{\re}{r_\ve}

 \newcommand{\lz}{L_{N}}

\newcommand{\ra}{\rightarrow}
\newcommand{\ve}{\varepsilon}

\newcommand{\mq}{m_q}
\newcommand{\iqt}{\int_{Q_T}}
\newcommand{\fp}{f^\prime(r)}
\newcommand{\fr}{f(r)}

\newcommand{\pt}{\partial_t}

\newcommand{\RN}{\mathbb{R}^N}

\newcommand{\irn}{{\int_{\RN}}}

\newcommand{\vps}{\varphi_i}

\newcommand{\vpss}{\varphi_i}

\newcommand{\ar}{A_{ r}}

\newcommand{\vp}{\varphi}

\newcommand{\rn}{\mathbb{R}^N}

\newcommand{\irnt}{\int_{ Q_T}}

\newcommand{\pxi}{\partial_{x_i}}
\newcommand{\pxj}{\partial_{x_j}}
\newcommand{\vo}{v^{(0)}}

\newtheorem{theorem}{Theorem}[section]

\newtheorem{lemma}[theorem]{Lemma}
\newtheorem{clm}[theorem]{Claim}

\theoremstyle{definition}
\newtheorem{definition}[theorem]{Definition}

\title[the incompressible Navier-Stokes equations in high space dimensions
] 
      {Large data  existence of global-in-time strong solutions to the incompressible Navier-Stokes equations in high space dimensions}

\author[Xiangsheng Xu]{}
 \email{xxu@math.msstate.edu}

\begin{document}

\maketitle

\centerline{\scshape Xiangsheng Xu}
\medskip
{\footnotesize
 \centerline{Department of Mathematics \& Statistics}
   \centerline{Mississippi State University}
   \centerline{ Mississippi State, MS 39762, USA}
} 

	\begin{abstract}We study the existence of a strong solution to the initial value problem for the incompressible Navier-Stokes equations in the whole space. Our investigation shows that  a local in-time strong solution to the problem  never develops singularity  whenever the initial velocity is divergence free and uniformly bounded with finite energy. As a result, it can be extended for all time. Our results seem to have given a positive answer to  the Navier-Stokes millennium problem proposed by the Clay Mathematical Institute.

\end{abstract}
\bigskip



\section{Introduction}
In this paper we investigate the boundedness of solutions to the initial value problem for the 
incompressible Navier-Stokes equations in the whole space $\rn$, $N\geq 3$.
Precisely, the problem reads
\begin{eqnarray}
	\pt v+(v\cdot\nabla) v+\nabla p&=&\nu\Delta v\ \ \mbox{in $\rn\times(0,T)\equiv Q_T$},\label{ns1}\\
	\nabla\cdot v&=&0\ \ \mbox{in $ Q_T$},\label{ns2}\\
	v(x,0)&=& v^{(0)}(x)\ \ \mbox{on $\rn$,}\label{ns3}
\end{eqnarray}
with given $T, \nu\in (0,\infty)$ and the initial data $v^{(0)}(x)$. 
Physically, the problem describes the flow of a fluid occupying $\rn$. In this case,  $v=\left(v_1,\cdots, v_N\right)^T$ is the velocity of the fluid,  $p$ is the pressure of the fluid, and $\nu$ is the viscosity.

The incompressible Navier-Stokes equations occupy a significant position
in partial differential equations. They represent the most fundamental model for the flow of an incompressible viscous fluid. As a result, they have been studied extensively in the literature. See, for example, \cite{C1,C2,CF,OP,RRS,T}. In spite of that, the basic issue of the well-posedness of the initial value problem in three or higher space dimensions remains elusive. The objective of this paper is to settle this well known open problem \cite{F}.



Our main result is the following theorem
\begin{theorem}\label{thm}Assume 
	that
	\begin{equation}\label{ns31}
		|v^{(0)}|\equiv\sqrt{\left(v^{(0)}_1\right)^2+\cdots+\left(v^{(0)}_N\right)^2}\in L^2(\rn)\cap L^\infty(\rn) \ \mbox{and $\nabla\cdot v^{(0)}=0$}.
	\end{equation}
	Let $(v, p)$ be a local (in-time) strong solution to \eqref{ns1}-\eqref{ns3}. 
	Then 
	there exist a positive number $c=c\left(N, \nu, \|\vo\|_{2,\rn}, \|\vo\|_{\infty,\rn}\right)$ such that
	\begin{equation}\label{ns24}
		\|v\|_{\infty,Q_T}\leq c\left(N, \nu, \|\vo\|_{2,\rn}, \|\vo\|_{\infty,\rn}\right).
	\end{equation}
	Here 
	$$\|\vo\|_{\ell,\rn}=\left(\int_{\rn}|\vo|^\ell dx\right)^{\frac{1}{\ell}}.$$
\end{theorem}
%

A result of \cite{OP} asserts that a local (in-time) strong solution to \eqref{ns1}-\eqref{ns3} does exist under \eqref{ns31}. Since the upper bound in \eqref{ns24} does not depend on $T$, our theorem asserts that a local (in-time) strong solution never develops singularity. As a result,  it can be extended for all time. This yields a positive answer to the Navier-Stokes millennium problem. 

The so-called 
global in time existence of a strong solution to \eqref{ns1}-\eqref{ns3} has been one of the most outstanding problems in partial differential equations \cite{F}, and many papers have been dedicated to the subject. Existing results can be classified into two categories. In the first category, one establishes the global existence by imposing some type of smallness assumptions on the initial velocity $v^{(0)}$, or the largeness of the viscosity $\nu$, or their combinations. This began with the work of Leray \cite{L}, who showed the global result by assuming 
\begin{equation*}
	\vo\in H^1(\rn)\ \ \mbox{and}\ \ \nu>>\|\vo\|^{\frac{1}{2}}\|_{2,\rn}\|\nabla\vo\|^{\frac{1}{2}}_{2,\rn}.
\end{equation*} 
Fujita and Kato \cite{FK} weakened the condition to 
$$\vo\in H^\frac{1}{2}(\rn)\ \ \mbox{and}\ \ \nu>>\|v^{(0)}\|_{H^{\frac{1}{2}}(\rn)}.$$
This result was further extended by Kato \cite{K} for the initial data in $L^N$, and by Cannone, Meyer
and Planchon \cite{CMP} for the initial data in the Besov space $B_{p,\infty}^{-1+\frac{3}{p}}$ with $p\in(3,\infty) $.  The
most recent result in this direction was given by Koch and Tataru \cite{KT}. They proved
the global existence under the condition that the initial velocity was sufficiently small in $BMO^{-1}$. 
We refer the reader to \cite{LZ} for other related works. We mention in passing that  Bourgain and Pavlovi\'{c} \cite{BP} showed the ill-posedness of \eqref{ns1}-\eqref{ns3} in $BMO^{-1}_{\infty,\infty}(\mathbb{R}^3)$.
In the second category, one obtains the existence of a strong solution by improving the integrability of a weak one. The most celebrated result in this direction is
the work of Serrin, Prodi, and Ladyzenskaja, which states that if a
weak solution $v$ belongs to 
$L^r((a,b);L^s)$
with $\frac{2}{r}+\frac{N}{s} \leq 1, s > N$, then $v$ is a
strong solution on the time interval (a, b); the critical case where $r = \infty, s = N$ was
proved by Escauriaza et al. \cite{ESS}. Note that when $s=r$ the condition becomes $v\in L^{N+2} $. A by-product of our development is that 
\begin{equation}\label{intr}
	|v|\in L^\infty(Q_T) \ \ \mbox{whenever}\ \ v\in L^{s}(Q_T)\ \ \mbox{for some $s>N+2$.}
\end{equation}
Finally, Beale et al. \cite{BKM} showed
that if $\mbox{curl}\ v \in  L^1
((a, b);L^\infty)$ then $v$ is a strong solution on the time interval
$(a, b]$. However, it is not clear if it is any easier to obtain these conditions.

%


The simultaneous presence of both $\nabla p$ and the term $(v\cdot\nabla) v$  in \eqref{ns1} seems to be the source of  all the major mathematical difficulties. Indeed, if the latter is absent, the resulting equations are the so-called Stokes equations and the existence of a classical solution to the initial value problem for such equations  has been known for a while \cite{OP}. On the other hand, if the pressure $p$ is a given nice function,  the classical regularity theory for linear parabolic equations  asserts that for each $q>N+2$ there is a constant $c=c(N,\nu,q)$ such that
\begin{equation}\label{pi}
	\|v\|_{\infty, Q_T}\leq 2\sqrt{N}\|\vo\|_{\infty,\rn}+c(N,\nu,q)\|\vo\|_{2,\rn}+2\sqrt{N}\|p\|_{q, Q_T}.
\end{equation}
Indeed,  the classical De Giorgi iteration scheme is still applicable in spite of the nonlinear term $(v\cdot\nabla) v$. 
Let us recall this argument. 
 For each $i\in \{1,\cdots,N\}$, the $i$-th component of the vector equation \eqref{ns1} is the scalar equation
\begin{eqnarray}
	\pt v_i+(v\cdot\nabla) v_i-\nu\Delta v_i&=&-\partial_i p\ \ \mbox{in $ Q_T$}.\label{ni1}
	\end{eqnarray}
It is coupled with  the initial condition
\begin{equation}
	v_i(x,0)= v_i^{(0)}(x)\ \ \mbox{on $\rn$}.\nonumber
\end{equation}
Select
\begin{equation}
	k\geq \|v_i^{(0)}\|_{\infty,\rn}\nonumber
\end{equation}
as below.
Define
\begin{equation}
	k_n=2k-\frac{k}{2^{n}}\ \ \mbox{for $n=0,1,2\cdots$.}\nonumber
\end{equation}
Subsequently, use $\left(v_i-k_{n+1}\right)^+$ as a test function in \eqref{ni1} to obtain
\begin{eqnarray}
\lefteqn{\frac{1}{2}\frac{d}{dt}\irn \left[\left(v_i-k_{n+1}\right)^+\right]^2dx+	\irn(v\cdot\nabla) v_i(v_i-k_{n+1})^+dx+\nu\irn\left|\nabla\left(v_i-k_{n+1}\right)^+\right|^2dx}\nonumber\\
&	=&\irn p\partial_i\left(v_i-k_{n+1}\right)^+dx\leq \frac{\nu}{2}\irn\left|\nabla\left(v_i-k_{n+1}\right)^+\right|^2dx+\frac{\nu}{2}\int_{\{v_i\geq k_{n+1}\}}p^2dx.\label{vi1}
\end{eqnarray}
The key observation is that \eqref{ns2} implies
\begin{equation}\label{vv}
	\irn(v\cdot\nabla) v_i(v_i-k_{n+1})^+dx =0.
\end{equation}
Use this in \eqref{vi1} and then integrate to derive
\begin{equation}
	\sup_{0\leq t\leq T}\irn\left[\left(v_i-k_{n+1}\right)^+\right]^2dx+\iqt\left|\nabla\left(v_i-k_{n+1}\right)^+\right|^2dxdt\leq c(\nu)\int_{\{v_i\geq k_{n+1}\}}p^2dxdt.\label{vi2}
\end{equation}
Set
\begin{equation}
	y_n=\iqt\left[\left(v_i-k_{n}\right)^+\right]^{2\lz}dxdt,\nonumber
\end{equation}
where
\begin{equation}\label{lzd}
	\lz=\frac{N+2}{N}.
\end{equation}
Subsequently,
\begin{equation}
	y_n\geq \left(k_{n+1}-k_n\right)^{2\lz}|\{v_i\geq k_{n+1}\}|=\frac{k^{2\lz}}{2^{2\lz(n+1)}}|\{v_i\geq k_{n+1}\}|.\nonumber
\end{equation}
This combined with \eqref{vi2} and \eqref{ns4} below yields
\begin{eqnarray}
	y_{n+1}&\leq &c(N,\nu)\left(\|p\|_{q, Q_T}^{2}|\{v_i\geq k_{n+1}\}|^{1-\frac{2}{q}}\right)^{\lz}\nonumber\\
	&\leq&\frac{c(N,\nu,q)\|p\|_{q, Q_T}^{2\lz}2^{2\lz(1+\alpha)n}}{k^{2\lz(1+\alpha)}}y_n^{1+\alpha},\label{vi3}
\end{eqnarray}
where
\begin{equation}
	\alpha=\left(1-\frac{2}{q}\right)\left(\frac{2}{N}+1\right)-1=\frac{2(q-N-2)}{Nq}>0.\nonumber
\end{equation}
We further require
\begin{equation}
	k\geq \|p\|_{q, Q_T}.\nonumber
\end{equation}
As a result, \eqref{vi3} is reduced to
\begin{equation}
		y_{n+1}\leq\frac{c(N,\nu,q)2^{2\lz(1+\alpha)n}}{k^{2\lz\alpha}}y_n^{1+\alpha}.\nonumber
\end{equation}
Now we are in a position to apply Lemma \ref{ynb} below. Upon doing so, we obtain
\begin{equation}
	\lim_{n\ra \infty}y_n=\iqt\left[\left(v_i-2k\right)^+\right]^{2\lz}dxdt=0,\nonumber
\end{equation}
provided that
\begin{equation}
	y_0=\iqt\left[\left(v_i-k\right)^+\right]^{2\lz}dxdt\leq\iqt\left(v_i^+\right)^{2\lz}dxdt\leq \frac{k^{2\lz}}{\left[c(N,\nu,q)\right]^\frac{1}{\alpha}2^{\frac{2\lz(1+\alpha)}{\alpha^2}}}.\nonumber
\end{equation}
Therefore, if we take
\begin{equation}
	k= \|v_i^{(0)}\|_{\infty,\rn}+\|p\|_{q, Q_T}+\left[c(N,\nu,q)\right]^\frac{1}{2\lz\alpha}2^{\frac{1+\alpha}{\alpha^2}}\|v_i^+\|_{2\lz, Q_T},\nonumber
\end{equation}
then we have
\begin{equation}
	\|v_i^+\|_{\infty, Q_T}\leq 2k.\nonumber
\end{equation}
As usual, we will denote the constant coefficient in the expression for $k$ by $c(N,\nu,q)$.
By applying the same argument to $-v_i$, we arrive at
\begin{equation}
	\|v_i^-\|_{\infty, Q_T}\leq  2\|v_i^{(0)}\|_{\infty,\rn}+2\|p\|_{q, Q_T}+c(N,\nu,q)\|v_i^-\|_{2\lz, Q_T}.\nonumber
\end{equation}
This together with \eqref{ns5} implies \eqref{pi}.
Our proof clearly indicates
\begin{equation}\label{li1}
	\lim_{q\ra(N+2)^+}c(N,\nu,q)=\infty.
\end{equation}

The classical way of eliminating $p$ from \eqref{ns1} is to take the curl of \eqref{ns1}. That is, one consider the equation satisfied by the vorticity $w=\nabla\times v$. When the space dimension is $2$, this method is very successful and leads to the existence of classical solutions. However, if $N\geq 3$ the method largely fails due to the so-called vortex stretching. 

The issue is that the pressure $p$ roughly behaves like $|v|^2$. 
To see this, we note from \eqref{ns2}
that
\begin{eqnarray}
	(v\cdot\nabla) v=(v_jv_i)_{x_j}.\nonumber
\end{eqnarray}
Here we have employed the notation convention of summing over repeated indices.	
Take the divergence of \eqref{ns1} to get
\begin{equation}\label{pdi}
	-\Delta p=(v_iv_j)_{x_ix_j}.
\end{equation}
As a result,
 \eqref{ns21} below holds,  and \eqref{pi} becomes
  \begin{equation}\label{pi1}
 	\|v\|_{\infty, Q_T}\leq 2\sqrt{N}\|\vo\|_{\infty,\rn}+c(N,\nu,q)\|\vo\|_{2,\rn}+c(N,q)\|v\|^2_{2q, Q_T}.
 \end{equation} 
 The constant $c(N,q)$ is determined by Lemma \ref{cz} below. Therefore,
 \begin{equation}\label{li2}
 	\lim_{q\ra \infty}c(N,q)=\infty.
 \end{equation}
Unfortunately, for $|v|^s$ to be integrable in a weak solution to \eqref{ns1}-\eqref{ns3}  $s$ is rather small. As indicated in \eqref{ns5} below, we only have
\begin{eqnarray}
	|v|\in L^{2\lz}(Q_T).\label{2lz}
\end{eqnarray}
It is also far from meeting the condition in the result of Serrin, Prodi, and Ladyzenskaja we mentioned earlier. 
On the other hand, if the power of $\|v\|_{2q, Q_T}$ in \eqref{pi1} had been $1$, \eqref{ns24} would follow from a suitable application of the interpolation inequality
\begin{eqnarray}\label{int}
	\|v\|_{2q, Q_T}\leq \|v\|_{\infty, Q_T}^{\frac{q-\lz}{q}}\|v\|_{2\lz, Q_T}^{\frac{\lz}{q}}.
\end{eqnarray}
This would be true for any $q\in(N+2, \infty)$. We are motivated to pursue the idea of trying to reduce the power of $\|v\|_{2q, Q_T}$ in \eqref{pi1}.  To this end, we  appeal to a  De Giorgi type of arguments. To elaborate further, 
let $w$ be the solution to the initial value problem
\begin{eqnarray}
	\pt w+(v\cdot\nabla)w-\nu\Delta w&=&0\ \ \mbox{in $Q_T$},\label{w1}\\
	w(x,0)&=&\vo(x)\ \ \mbox{on $\rn$}.\label{w2}
\end{eqnarray}
Set
\begin{equation}\label{uvw}
	u= v-w
\end{equation}
 Then $u$ satisfies 
\begin{eqnarray}
	\pt u+(v\cdot\nabla)u-\nu\Delta u&=&-\nabla p\ \ \mbox{in $Q_T$},\label{u1}\\
	u(x,0)&=&0\ \ \mbox{on $\rn$}.\label{u2}
\end{eqnarray}
The new idea here is \eqref{kcon6} below and its subsequent applications. Roughly speaking, by dividing through \eqref{u1} by $\|u\|_{r,Q_T}$ for a suitable choice of $r$ and then making use of \eqref{kcon6} in a proper manner, we are able to establish \eqref{hh1}. That is to say, we have decomposed the $L^{2q}$- norm squared in the last term of \eqref{pi1} into a product of three different norms, each of which carries an  exponent. The redeeming feature is that the total sum of their exponents can be made less than or equal to $1$ as indicated in \eqref{tot} below.
However,  we must be able to combine these three norms into a single one without increasing the total sum  too much. To illustrate, suppose that we wish to transform
\begin{equation}\label{v1}
\|u\|_{\ell,Q_T}^{\beta}\|u\|_{r,Q_T}^{\alpha}	
\end{equation}
into a single norm. If $\alpha>0$ and $\ell>r>2\lz$, we can eliminate the $L^r$-norm by invoking the interpolation inequality
\begin{equation}
	\|u\|_{r,Q_T}\leq \|u\|_{\ell,Q_T}^{\frac{\ell(r-2\lz)}{r(\ell-2\lz)}}\|u\|_{2\lz,Q_T}^{\frac{2\lz(\ell-r)}{r(\ell-2\lz)}}.\label{int1}
\end{equation}
This leads to
\begin{equation}
	\|u\|_{\ell,Q_T}^{\beta}\|u\|_{r,Q_T}^{\alpha}\leq \|u\|_{\ell,Q_T}^{\beta+\frac{\ell(r-2\lz)\alpha}{r(\ell-2\lz)}}\|u\|_{2\lz,Q_T}^{\frac{2\lz(\ell-r)\alpha}{r(\ell-2\lz)}}.\nonumber
\end{equation}
In view of \eqref{2lz}, we may treat the $L^{2\lz}$-norm as a given constant. 
The new exponent of the $L^\ell$-norm satisfies
\begin{equation}
\beta+\frac{\ell(r-2\lz)\alpha}{r(\ell-2\lz)}<\alpha+\beta.\nonumber
\end{equation}
We also see that  when we increase $r$ in the term $\|u\|_{r, Q_T}^\alpha$ it will make the exponent
$\alpha$ smaller.

If $\beta<0$ and $\ell>r$, we raise both sides of \eqref{int1} to the power of $-\frac{r(\ell-2\lz)\beta}{\ell(r-2\lz)}$ to obtain
\begin{equation}
	\|u\|_{r,Q_T}^{-\frac{r(\ell-2\lz)\beta}{\ell(r-2\lz)}}\leq \|u\|_{\ell,Q_T}^{-\beta}\|u\|_{2\lz,Q_T}^{-\frac{2\lz(\ell-r)\beta}{\ell(r-2\lz)}}.\nonumber
\end{equation}
Incorporate this into \eqref{v1} to deduce
\begin{equation}
\|u\|_{\ell,Q_T}^{\beta}\|u\|_{r,Q_T}^{\alpha}\leq 	\|u\|_{r,Q_T}^{\alpha+\frac{r(\ell-2\lz)\beta}{\ell(r-2\lz)}}\|u\|_{2\lz,Q_T}^{-\frac{2\lz(\ell-r)\beta}{\ell(r-2\lz)}}.\nonumber
\end{equation} 
Once again, we have
\begin{equation}
	\alpha+\frac{r(\ell-2\lz)\beta}{\ell(r-2\lz)}<\alpha+\beta.\nonumber
\end{equation} The real challenge is how to combine the  two norms in \eqref{v1} when
\begin{equation}
\beta>0,\ \ \alpha<0, \ \ \mbox{and}	\ \ \ell>r.\nonumber
\end{equation} 
To address the third case, we consider the  function
\begin{equation}\label{fd}
	f(r)=\ln\left(\iqt |u|^rdxdt\right)\ \ \mbox{for $r\geq2\lz$}.
\end{equation}
It turns out that this function is convex. Moreover,
\begin{equation}\label{ffl}
	\lim_{r\ra\infty}\frac{\fr}{r}=	\lim_{r\ra\infty}\fp=\ln\left(\|u\|_{\infty, Q_T}\right).
\end{equation}
We may express the product in \eqref{v1} as
\begin{equation}
\|u\|_{\ell,Q_T}^{\beta}\|u\|_{r,Q_T}^{\alpha}	=e^{\frac{\beta f(\ell)}{\ell}+\frac{\alpha f(r)}{r}}.	\nonumber
\end{equation}
We solve the third case via suitable applications of various properties of $\fr$. This constitutes the core of our development. In a way, we can summarize our approach as follows:
\begin{equation}
	\|u\|_{2q, Q_T}^2\leq c\|u\|_{\ell, Q_T}^{\gamma_0}\|u\|_{r, Q_T}^{\beta_0}\|u\|_{2q, Q_T}^{\alpha_0}\leq c\|u\|_{r, Q_T}^\theta,\nonumber
\end{equation}
where 
\begin{eqnarray}
	&&\ell>r>2q,  \nonumber\\
&&	\gamma_0<0,\  \beta_0>0, \ \alpha_0 <0 \ \ \mbox{with}\ \ \alpha_0+\beta_0+\gamma_0\leq 1, \mbox{and} \nonumber\\
&&	\theta\in \left(0,\frac{r}{r-2\lz}\right).\label{the}
\end{eqnarray}
Thus, our basic strategy is to first decompose and then combine. The trick is how to achieve \eqref{the}.

Obviously, solutions to \eqref{ns1}-\eqref{ns3} are not unique in the $p$-component. If $(v,p)$ is a solution, so is $(v,p+g(t))$ for any function $g(t)$. We follow the tradition \cite{OP} to represent $p$ as a Newtonian potential (\cite{GT}, p. 18), which implies that
$$\irn p(x,t)dx=0.$$
The representation also enables us to employ certain properties of the Calder\'{o}n-Zygmund kernel \cite{CFL,CFL1}.


This work is organized as follows. In Section \ref{sec2}, we collect some relevant known results, while Section \ref{sec3} is devoted to the proof of Theorem \ref{thm}.

\section{Preliminary results}\label{sec2} In this section, we collect a few relevant known results.

The following lemma can be found in (\cite{D}, p.12).
\begin{lemma}\label{ynb}
	Let $\{y_n\}, n=0,1,2,\cdots$, be a sequence of positive numbers satisfying the recursive inequalities
	\begin{equation*}
		y_{n+1}\leq cb^ny_n^{1+\alpha}\ \ \mbox{for some $b>1, c, \alpha\in (0,\infty)$.}
	\end{equation*}
	If
	\begin{equation*}
		y_0\leq c^{-\frac{1}{\alpha}}b^{-\frac{1}{\alpha^2}},
	\end{equation*}
	then $\lim_{n\rightarrow\infty}y_n=0$.
\end{lemma}

We also need some results from \cite{CFL,CFL1}.

\begin{definition}
	A function $k(x)$ on $\RN\setminus\{0\}$ is called a Calder\'{o}n-Zygmund kernel (in short, C-Z kernel) if:
	\begin{enumerate}
		\item[(i)] $k\in C^\infty(\RN\setminus\{0\})$;
		\item[(ii)]$k(x)$ is homogeneous of degree $-N$, i.e., $k(tx)=t^{-N}k(x)$;
		\item[(iii)] $\int_{\partial B_1(0)}k(x)d\mathcal{H}^{N-1}=0$.
	\end{enumerate}
\end{definition}
The most fundamental result concerning C-Z kernels \cite{CFL} is the following
\begin{lemma}
	Given a C-Z kernel $k(x)$,  we define
	\begin{equation}
		\mathcal{K}_\varepsilon f(x)=\int_{\RN\setminus B_\varepsilon(x)}k(x-y)f(y)dy\ \ \mbox{for $\varepsilon>0$ and $f\in L^q(\RN)$ with $q\in (1,\infty)$.}\nonumber
	\end{equation}
	Then:
	\begin{enumerate}
		\item[\textup{(CZ1)}] For each $f\in L^q(\RN)$ there exists a function $\mathcal{K}f\in L^q(\RN)$ such that
		\begin{equation}
			\lim_{\varepsilon\rightarrow 0}\|\mathcal{K}_\varepsilon f-\mathcal{K}f\|_{q,\RN}=0.\nonumber
		\end{equation}
		In this case we use the notation
		\begin{equation}
			\mathcal{K}f(x)=\textup{P.V.}k*f(x)=\textup{P.V.}\int_{\RN}k(x-y)f(y)dy.\nonumber
		\end{equation}
		\item[\textup{(CZ2)}] The operator $\mathcal{K}$ is bounded  on $ L^q(\RN)$. More precisely, we have
		\begin{equation}
			\|\mathcal{K}f\|_{q,\RN}\leq c\left(\int_{\partial B_1(0)}k^2(x)d\mathcal{H}^{N-1}\right)^\frac{1}{2}	\|f\|_{q,\RN},\nonumber
		\end{equation}
		where the positive number $c$ depends only on $N, q$.
	\end{enumerate}
\end{lemma}
Finally, the following two inequalities will be used without acknowledgment:
\begin{eqnarray*}
	(|a|+|b|)^\gamma&\leq&\left\{\begin{array}{ll}
		2^{\gamma-1}(|a|^\gamma+|b|^\gamma)&\mbox{if $\gamma\geq 1$},\\
		|a|^\gamma+|b|^\gamma&\mbox{if $\gamma\leq 1$}.
	\end{array}\right.
\end{eqnarray*}
Unless otherwise stated, the letter $c$  denotes a generic positive number, which only depends on $N, \nu$, and the various parameters we introduce. In particular, it is independent of $T$ and $\vo$.
\section{Proof of Theorem \ref{thm}}\label{sec3}
We first would like to point out that a strong solution  over a time interval is also smooth there \cite{OP}.
Therefore, in our subsequent calculations we may assume that $(u,p)$ is a classical solution. The nature of our argument is to turn a qualitative assumption into a quantitative estimate. To be precise, we show that if
\begin{equation}\label{uub}
	|v|\in L^\infty(Q_T),
\end{equation} then \eqref{ns24} must be true. The existence of a local-in-time classical solution is known under \eqref{ns31} due to a result in \cite{OP}. Here we show that such a solution never blows up. As a result, it can be extended as a global strong solution. 

Before we start the proof of Theorem \ref{thm}, we introduce a few lemmas. They are largely known,
and we include them here for completeness.%
\begin{lemma} We have 
	\begin{equation}\label{ns5}
		\|v\|_{2\lz,Q_T}\leq c(N,\nu)\|\vo\|_{2,\rn},
	\end{equation}
	where $\lz$ is given as in \eqref{lzd}.
\end{lemma}
\begin{proof}
In view of Theorem 3.2 in \cite{OP}, we may	 use $v$ as a test function in \eqref{ns1}. Upon doing so,  we derive
\begin{equation}\label{ns}
		\frac{1}{2}\frac{d}{dt}\irn|v(x,t)|^2dx+\irn(v\cdot \nabla)v \cdot v dx+\nu\irn|\nabla v|^2dx=0.
\end{equation}
Here we have used \eqref{ns2}. Use it again in the second term to obtain
\begin{equation}
	\irn(v\cdot \nabla)v \cdot v dx=\frac{1}{2}\irn (v\cdot \nabla)|v|^2dx= \frac{1}{2}\irn v\cdot \nabla|v|^2dx=0.\nonumber
\end{equation}
Collect this in \eqref{ns} and then integrate to deduce
\begin{equation}\label{ns7}
	\frac{1}{2}\irn|v(x,t)|^2dx+
	\nu\int_{0}^{t}\irn|\nabla v|^2dxd\tau= \frac{1}{2}
	\irn|v^{(0)}|^2dx.
\end{equation}
Recall 	the Sobolev inequality in the whole space to obtain
\begin{equation*}
	\|f\|_{\frac{2N}{N-2},\rn}\leq c(N)\|\nabla f\|_{2,\rn}\ \ \mbox{for each $f\in H^1(\rn)$}.
\end{equation*}
This together with  \eqref{ns7} implies that
\begin{eqnarray}
\int_{Q_T}|v|^{2\lz}dxdt&=&	\int_{Q_T}|v|^{\frac{4}{N}+2}dxdt\nonumber\\
&\leq&\int_{0}^{T}\left(\irn|v|^2dx\right)^{\frac{2}{N}}\left(\irn|v|^{\frac{2N}{N-2}}dx\right)^{\frac{N-2}{N}}dt\nonumber\\
	&\leq&\left(\sup_{0\leq t\leq T}\irn|v|^2dx\right)^{\frac{2}{N}}\int_{0}^{T}\left(\irn|v|^{\frac{2N}{N-2}}dx\right)^{\frac{N-2}{N}}dt\nonumber\\
	&\leq&c(N)\left(\sup_{0\leq t\leq T}\irn|v|^2dx\right)^{\frac{2}{N}}\int_{0}^{T}\irn|\nabla v|^{2}dxdt\nonumber\\
	&\leq& c(N,\nu)\left(\irn|v^{(0)}|^2dx\right)^{\frac{2}{N}+1},\label{ns4}
\end{eqnarray}
from which the lemma follows. 
\end{proof}

Now we turn our attention to the pressure $p$.
\begin{lemma}\label{cz}For each $ s  >1$ there is a positive number $c=c(s,N)$  with
	\begin{equation}\label{ns21}
		\|p\|_{  s  ,\rn}\leq c(s,N)  \|v\|_{2  s  ,\rn}^2.
	\end{equation}
\end{lemma}
\begin{proof}
	It follows from \eqref{ns2} that
	\begin{equation*}
		\nabla\cdot\left[	(v\cdot\nabla) v\right]=\pxi\left(v_j\pxj v_i\right)=\pxi\left[\pxj (v_jv_i)-v_i\pxj v_j\right]=\partial^2_{x_ix_j}(v_jv_i).
	\end{equation*}
	In view of \eqref{pdi}, we may invoke
	 the classical representation theorem (\cite{GT}, p. 17), thereby obtaining
	\begin{equation*}
		p(x,t)=\irn\Gamma(y-x)\partial^2_{y_iy_j}(v_jv_i)dy=\irn\partial^2_{y_iy_j}\Gamma(y-x)v_jv_idy,
	\end{equation*}
	where $\Gamma (x)$ is the fundamental solution of the Laplace equation, i.e.,
	$$\Gamma(x)=\frac{1}{N(N-2)\omega_N |x|^{N-2}},\ \ \omega_N= \mbox{the volume of the unit ball in $\rn$.}$$
	It is a well known fact that $\partial^2_{y_iy_j}\Gamma(y)$ is a Calder\'{o}n-Zygmund kernel. Thus, (CZ2) asserts that for each $  s  \in (1,\infty)$ there is a positive number $c=c(s,N) $  such that \eqref{ns21} holds. The proof is complete.
\end{proof}

Let $w$ be the solution of \eqref{w1}-\eqref{w2}. 
We easily infer from \eqref{vv} that the weak maximum principle remains true for solutions to \eqref{w1}. Hence, 
\begin{equation}\label{wm}
	\|w_i\|_{\infty, Q_T}\leq \|\vo_i\|_{\infty,\rn},\ \ i=1,\cdots, N.
\end{equation}
Appealing to the proof of \eqref{ns5}, we arrive at 
	\begin{equation}\label{ns6}
	\|w\|_{2\lz,Q_T}\leq c(N,\nu)\|\vo\|_{2,\rn}.
\end{equation}
Combine the preceding two estimates with \eqref{uvw} to obtain 
\begin{equation}\label{vw}
	\|u\|_{2\lz,Q_T}\leq c(N,\nu)\|\vo\|_{2,\rn}\ \ \mbox{and}\ \ \|u\|_{\infty,Q_T}\leq \|v\|_{\infty,Q_T}+\sqrt{N}\|\vo\|_{\infty,\rn}.
\end{equation}
In view of \eqref{uub} and \eqref{int}, we have
\begin{equation}
		|u|\in L^{r}(Q_T)\ \ \mbox{for each $r\geq 2\lz$.}\nonumber
\end{equation}
Let $\fr$ be given as in \eqref{fd}.
We can easily verify that $	f(r)$ is a convex function on $(2\lz, \infty)$. Indeed, let $2\lz<r_1<r_2$ and $\lambda\in [0,1]$. The interpolation inequality asserts
\begin{equation}
	\|u\|_{\lambda r_1+(1-\lambda)r_2, Q_T}\leq 	\|u\|_{ r_1, Q_T}^{\frac{\lambda r_1}{\lambda r_1+(1-\lambda)r_2}}	\|u\|_{r_2, Q_T}^{\frac{(1-\lambda)r_2}{\lambda r_1+(1-\lambda)r_2}}.\nonumber
\end{equation}
Raise both sides to the power of $\lambda r_1+(1-\lambda)r_2$ and then take logarithm  to derive
\begin{equation}
	f(\lambda r_1+(1-\lambda)r_2)\leq \lambda f(r_1)+(1-\lambda)f(r_2).\nonumber
\end{equation}
Remember that $u$ is a very ``nice'' function. We may assume that $\fr$ is a smooth function. As a result,  the convexity of $f$ can also be established by computing
\begin{eqnarray}
	f^\prime(r)&=&\frac{\iqt |u|^r\ln |u|dxdt}{\iqt |u|^rdxdt},\label{fp}\\
	f^{\prime\prime}(r)&=&\frac{\iqt |u|^r\ln^2 |u|dxdt\iqt |u|^rdxdt-\left(\iqt |u|^r\ln |u|dxdt\right)^2}{\left(\iqt |u|^rdxdt\right)^2}.\nonumber
\end{eqnarray}
Apply H\"{o}lder inequality in the expression for $f^{\prime\prime}(r)$ to derive
\begin{equation}\label{pp}
	f^{\prime\prime}(r)\geq 0\ \ \mbox{for $r\in (2\lz,\infty)$}.
\end{equation}

We are ready to verify \eqref{ffl}.
Obviously,  we may conclude from \eqref{int} that
\begin{equation}\label{mm9}
	\limsup_{r\ra\infty}\|u\|_{r,Q_T}\leq  \|u\|_{\infty,Q_T}.
\end{equation}
On the other hand, we have
\begin{equation}
	\left|\{|u|\geq \|u\|_{\infty,Q_T}-\ve\}\right|>0\ \ \mbox{for $\ve\in\left(0, \|u\|_{\infty,Q_T}\right)$.}\nonumber
\end{equation}
Subsequently,
\begin{equation}
	\|u\|_{r,Q_T}\geq\left(\left(\|u\|_{\infty,Q_T}-\ve\right)^{r}	\left|\{|u|\geq \|u\|_{\infty,Q_T}-\ve\}\right|\right)^{\frac{1}{r}}\ra \|u\|_{\infty,Q_T}-\ve\ \ \mbox{as $r\ra\infty$.}	\nonumber
\end{equation}
This together with \eqref{mm9} implies the first limit in \eqref{ffl}.
To see the second limit there, we obtain from \eqref{fp} that
\begin{eqnarray}\label{fpb}
		f^\prime(r)\leq \ln\left(\|u\|_{\infty, Q_T}\right).
\end{eqnarray}
On the other hand, the result \eqref{pp}, the convexity of $\fr$, implies
\begin{eqnarray}
	\fp\geq\frac{\fr-f(2\lz)}{r-2\lz}=\frac{\fr}{r}\frac{r}{r-2\lz}-\frac{f(2\lz)}{r-2\lz}\ra \ln\left(\|u\|_{\infty, Q_T}\right)\ \ \mbox{as $r\ra \infty$}.\label{fpi}
\end{eqnarray}
 Combining this with \eqref{fpb} yields \eqref{ffl}.

\begin{proof}[Proof of Theorem \ref{thm}] Define
	\begin{equation}\label{msrd}
		\ar=\|u\|_{r, Q_T}\ \ \mbox{for $r\in\left(2\lz,\infty\right)$.}
	\end{equation}
		Subsequently, let
	\begin{equation}\label{pdef}
	\varphi=\frac{u}{\ar}.
	\end{equation}
	Denote by $\varphi_i$ the $i$-th component of $\varphi$.
	Then we have
	\begin{equation}\label{dec3}
		\|\varphi_i\|_{r, Q_T}\leq 	\|\varphi\|_{r, Q_T}=1\ \ \mbox{for each $i\in\{1,\cdots,N\}$}.
	\end{equation}
	Divide through \eqref{u1} by $\ar$ to obtain
		\begin{equation}\label{jan1}
		\pt \vps+ (v\cdot\nabla) \vps-\nu\Delta \vps=-\ar^{-1}\partial_i p\ \ \mbox{in $Q_T$,}\ \ i=1,\cdots, N.
	\end{equation}
We shall employ a De Giorgi-type iteration scheme. Without any loss of generality, we may assume that
\begin{equation}\label{pii}
	\|\vps\|_{\infty, Q_T}=	\|\vps^+\|_{\infty, Q_T}.
\end{equation}
Otherwise, consider $-\vps$.
	Select 
	\begin{equation}
		k>0\nonumber
		\end{equation}
	as below.	Define
	\begin{eqnarray}
		k_n&=&2k-\frac{k}{2^{n}} \ \ \mbox{for $n=0,1,\cdots$.}\nonumber
	\end{eqnarray}
Then  the function
$$\Psi_{n}\equiv\left(\ln\vpss-\ln k_n\right)^+$$
is a legitimate test function  for \eqref{jan1}. 
Indeed, it is elementary to check that for each $\sigma\in(0,1]$ we have
$$\Psi_{n}\leq \frac{1}{\sigma k_n}\left(\vpss-k_n\right)^+.$$
This together with \eqref{pdef} implies
$$\Psi_{n}\in L^s(Q_T)\ \ \mbox{for each $s\geq 2\lz$}.$$
Set
\begin{eqnarray}
	\Omega_{n}(t)&=&	\{x\in\rn:\vps(x,t)\geq k_{n}\},
		\nonumber\\
		Q_{n}&=&\{(x,t)\in  Q_T: \vps(x,t)\geq k_{n}\}
		.\label{jan5}
\end{eqnarray} 
Subsequently,
$$|\nabla \Psi_{n}|=\left|\frac{1}{ \vpss }\nabla \vpss \chi_{Q_{n }}\right|\leq\frac{|\nabla v_i|}{ \ar k_n }.$$
We are ready to use $\Psi_{n} $ as a test function in \eqref{jan1}. Upon doing so, we obtain
\begin{eqnarray}
	\lefteqn{	\frac{d}{dt}\irn\int_{k_{n}}^{\vps}\left(\ln\mu-\ln k_n\right)^+d\mu   dx+\nu\irnn\frac{1}{ \vpss }\left|\nabla\vpss\right|^2 dx}\nonumber\\
	&=&-\irn (v\cdot\nabla) \vpss \Psi_{n} dx+\ar^{-1}\irn p\partial_i\Psi_{n} dx
.\label{ub1}
\end{eqnarray}
	Note from \eqref{ns2} that
	\begin{eqnarray}
		\irnn (v\cdot\nabla) \vpss\Psi_{n}dx&=&	\irn v\cdot\nabla \int_{k_{n}}^{\vpss}\left(\ln\mu-\ln k_n\right)^+d\mu dx=0.\nonumber
	\end{eqnarray}
As for the last term in \eqref{ub1}, we have
	\begin{eqnarray}
\ar^{-1}\irn p\partial_i\Psi_{n} dx
		&\leq&\ar^{-1}k_n^{-\frac{1}{2}}\irnn |p|\vpss^{-\frac{1}{2}}|\nabla\vpss| dx\nonumber\\
		&\leq&\frac{\nu}{4}\irnn\frac{1}{ \vpss }\left|\nabla\vpss\right|^2dx+\frac{\ar^{-2}k_n^{-1}}{\nu}\irnn p^2dx.
		\nonumber
			\end{eqnarray}
			Collect the preceding two results in \eqref{ub1} to get
	\begin{eqnarray}
		\frac{d}{dt}\irn\int_{k_{n}}^{\vpss}\left(\ln\mu-\ln k_n\right)^+d\mu   dx+\frac{\nu}{2}\irnn\frac{1}{ \vpss }\left|\nabla\vpss\right|^2 dx	\leq\frac{\ar^{-2}k_n^{-1}}{\nu}\irnn p^2dx
		.\label{ub2}
	\end{eqnarray}	
Evidently,
\begin{eqnarray}
	\irnn\frac{1}{ \vpss }\left|\nabla\vpss\right|^2dx&=&4\irn\left|\nabla\left(\sqrt{\vpss}-\sqrt{k_n}\right)^+\right|^2dx.\label{ub4}
\end{eqnarray}
We next claim 
	\begin{equation}\label{hope10}
		\int_{k_{n}}^{\vpss}\left(\ln\mu-\ln k_n\right)^+d\mu\geq  2\left[\left(\sqrt{\vpss}-\sqrt{k_n}\right)^+\right]^2.
	\end{equation}
	To see this, we compute
	\begin{eqnarray*}
		\left(\int_{k_{n}}^{\vpss}\left(\ln\mu-\ln k_n\right)^+d\mu\right)^{\prime\prime}&=&\vpss^{-1}\chi_{Q_{n }},\\
		\left(2\left[\left(\sqrt{\vpss}-\sqrt{k_n}\right)^+\right]^2\right)^{\prime\prime}&=&k_n^{\frac{1}{2}}\vpss^{-\frac{3}{2}}\chi_{Q_{n }}.
	\end{eqnarray*}
	We can easily verify  that
	\begin{equation*}
		\left(\int_{k_{n}}^{\vpss}\left(\ln\mu-\ln k_n\right)^+d\mu\right)^{\prime\prime}\geq \left(2\left[\left(\sqrt{\vpss}-\sqrt{k_n}\right)^+\right]^2\right)^{\prime\prime}\ \ \mbox{in $Q_n$.}
	\end{equation*}
	Integrate this inequality twice and choose the constant of integration appropriately each time to obtain \eqref{hope10}.

Recall  \eqref{u2} to obtain
$$\left.\int_{k_{n}}^{\vpss}\left(\ln\mu-\ln k_n\right)^+d\mu\right|_{t=0}=0.$$
With this in mind, we integrate \eqref{ub2} with respect to $t$ and then collect  \eqref{ub4}, \eqref{wm}, and \eqref{hope10} in the resulting inequality to deduce
	\begin{eqnarray}
\lefteqn{\sup_{0\leq t\leq T}\irn\left[\left(\sqrt{\vpss}-\sqrt{k_n}\right)^+\right]^2dx}\nonumber\\
&&+\irnt\left|\nabla\left(\sqrt{\vpss}-\sqrt{k_n}\right)^+\right|^2dxdt
		\leq c(\nu)\ar^{-2}k_n^{-1}\int_{Q_{n }} p^2dxdt.\label{hop20}
	\end{eqnarray}
%
Set
\begin{equation}\label{ynd}
	y_n=|Q_n|.
\end{equation}
%
We proceed to show that $\{y_n\}$ satisfies the condition in Lemma \ref{ynb}. 
By calculations similar to those in \eqref{ns4},  we have
\begin{eqnarray}
	\lefteqn{	\irnt\left[\left( \sqrt{\vps}-\sqrt{k_n}\right)^+\right]^{2\lz}dxdt}\nonumber\\
	&\leq&\int_{0}^{T}\left(\irn\left[\left( \sqrt{\vps}-\sqrt{k_n}\right)^+\right]^{2}dx \right)^{\frac{2}{N}}\left(\irn\left[\left( \sqrt{\vps}-\sqrt{k_n}\right)^+\right]^{\frac{2N}{N-2}}dx\right)^{\frac{N-2}{N}}dt\nonumber\\
	&\leq& c(N)\left(\sup_{0\leq t\leq T}\irn\left[\left( \sqrt{\vps}-\sqrt{k_n}\right)^+\right]^{2}dx \right)^{\frac{2}{N}}\irnt\left|\nabla\left( \sqrt{\vps}-\sqrt{k_n}\right)^+\right|^2dxdt\nonumber\\
	&\leq& c(N,\nu)\left(\ar^{-2} k_n^{-1}\int_{Q_{n}} p^2dxdt\right)^{\lz}.\label{rub3}
\end{eqnarray}
Here we have applied \eqref{hop20} in the last step.
It is easy to verify that
\begin{eqnarray*}
	\irnt\left[\left( \sqrt{\vps}-\sqrt{k_n}\right)^+\right]^{2\lz}dxdt&\geq&
	\int_{Q_{n+1}}\left[\left( \sqrt{\vps}-\sqrt{k_n}\right)^+\right]^{2\lz}dxdt\nonumber\\
	&\geq&\left(\sqrt{k_{n+1}}-\sqrt{k_n}\right)^{2\lz}|Q_{n+1}|
	\nonumber\\
	\nonumber\\
	&	\geq&\frac{k^{\lz}|Q_{n+1}|}{2^{2\lz(n+3)}}.
\end{eqnarray*}
Combining this with \eqref{rub3} yields
\begin{equation}
	y_{n+1}=|Q_{n+1}|^{\frac{N}{N+2}+\frac{2}{N+2}}\leq \frac{c(N,\nu)4^{n}}{k^2\ar^{2}}\int_{Q_{n}} p^2dxdt|Q_{n+1}|^{\frac{2}{N+2}}\leq \frac{c(N,\nu)4^{n}}{k^2\ar^{2}}\int_{Q_{n}} p^2dxdty_n^{\frac{2}{N+2}}.\label{rub4}
\end{equation}
	Fix
\begin{eqnarray}
	q>N+2.\label{qcon}
\end{eqnarray}
Then we can conclude from \eqref{ns21} that
\begin{eqnarray*}
	\int_{Q_{n}}p^2dxdt&\leq&\left(\irnt|p|^qdxdt\right)^{\frac{2}{q}}|Q_{n}|^{1-\frac{2}{q}}
	\leq c(N,q)\|v\|_{2q,Q_T}^{4}y_n^{1-\frac{2}{q}}.
\end{eqnarray*}
Use this in \eqref{rub4} to derive
\begin{eqnarray}
	y_{n+1}
	&\leq &	\frac{c(N,\nu,q) 4^{n}\|v\|_{2q,Q_T}^{4}}{k^2\ar^2}y_n^{1+\alpha},\label{ns11}
\end{eqnarray}
where
\begin{eqnarray}
	\alpha&=&-\frac{2}{q}+\frac{2}{N+2} =\frac{2(q-N-2)}{q(N+2)}>0.\label{adef}
\end{eqnarray}
We introduce a new parameter
\begin{equation}\label{ldef}
	\ell>r.
\end{equation}
Subsequently,  take 
\begin{equation}\label{kcon6}
\max\left\{L_1 \|\vp\|_{\ell,Q_T}^{\frac{\ell}{\ell-r}},\ L_2\ar^{-1}\|u\|_{2q,Q_T}^{\frac{ q}{q-\lz}}\right\} \leq k, %
\end{equation}
where $L_1$ and $L_2$ are two positive numbers to be determined. Note that $k$ is roughly $\|\vp\|_{\infty,Q_T}$ and the exponent of $\|\vp\|_{\ell,Q_T}$ in the above inequality is so chosen that  
\begin{eqnarray}
	\|\vp\|_{\ell,Q_T}^{\frac{\ell}{\ell-r}}&\leq&\left[\|\vp\|_{\infty,Q_T}^{\frac{\ell-r}{\ell}}\|\vp\|_{r,Q_T}^{\frac{r}{\ell}}\right]^{\frac{\ell}{\ell-r}}\nonumber\\
	&\leq& \|\vp\|_{\infty,Q_T}\|\vp\|_{r,Q_T}^{\frac{r}{\ell-r}}=\|\vp\|_{\infty,Q_T}.\label{jm1}
\end{eqnarray}
The last step is due to \eqref{dec3}. The exponent of $\|v\|_{2q,Q_T}$ in \eqref{kcon6} is given with similar consideration in mind. 
As we shall see, the terms between the two big brackets in \eqref{kcon6} will be absorbed into either $\|\vp\|_{\infty,Q_T}$ or $\|v\|_{\infty,Q_T}$, leaving behind norms with negative powers. This is the key to our approach.

 We easily see from \eqref{kcon6} that for each $j>0$ there hold
\begin{equation}
	L_1^{\alpha j} \|\vp\|_{\ell,Q_T}^{\frac{\alpha j\ell}{\ell-r}}\leq k^{\alpha j}\ \ \mbox{and}\ \ 	L_2^{j\alpha+2}\ar^{-(j\alpha+2)}\|u\|_{2q,Q_T}^{\frac{(j\alpha+2)q}{q-\lz}}\leq k^{j\alpha+2}.\nonumber
\end{equation}
Collect these two inequalities in \eqref{ns11} to deduce
\begin{equation*}
	y_{n+1}\leq \frac{c(N,\nu,q) 4^n\ar^{j\alpha}\|v\|_{2q,Q_T}^{4} k^{2j\alpha}}{L_1^{\alpha j} \|\vp\|_{\ell,Q_T}^{\frac{\alpha j\ell}{\ell-r}}L_2^{j\alpha+2}\|u\|_{2q,Q_T}^{\frac{(j\alpha+2)q}{q-\lz}}} y_n^{1+\alpha}.
\end{equation*}
 Recall \eqref{ynd}, \eqref{jan5}, and \eqref{dec3} to deduce
 	\begin{eqnarray*}
 		y_0&=&
 		|Q_0|\leq \irnt\left(\frac{\vps^+}{k}\right)^{r}dxdt\leq\frac{1}{k^r}.
 	\end{eqnarray*}
 Take
 	\begin{equation}\label{rcon}
 		r>2j.
 	\end{equation}
 	Subsequently, we can pick $k$ so large that
 	\begin{eqnarray}
 		\frac{1}{k^{r-2j}}\leq \frac{L_1^{j}L_2^{\frac{j\alpha+2}{\alpha}}\|\vp\|_{\ell,Q_T}^{\frac{j\ell}{\ell-r}}\|u\|_{2q,Q_T}^{\frac{(j\alpha+2)q}{\alpha(q-\lz)}}}{ \left[c(N,\nu,q)\right]^{\frac{1}{\alpha}}4^{\frac{1}{\alpha^2}}\ar^{j}\|v\|_{2q,Q_T}^{\frac{4}{\alpha}}}.\label{jm7}
 	\end{eqnarray}
 	We are now in a position to apply Lemma \ref{ynb}. Upon doing so, we arrive at 
 	$$\lim_{n\ra \infty}y_n=|\{\vps\geq 2k\}|=0.$$ 
 This together with \eqref{pii} implies
 	\begin{equation}
 		\sup_{Q_T}	\|\vps\|_{\infty, Q_T}\leq 2k.\nonumber
 	\end{equation}
 	Subsequently,
 	\begin{equation}\label{jm8}
 		\sup_{Q_T}	\|\vp\|_{\infty, Q_T}\leq 2\sqrt{N}k.
 	\end{equation}
 	According to  \eqref{kcon6} and \eqref{jm7}, it is enough for us to take
 	\begin{eqnarray}
 		k&=&L_1 \|\vp\|_{\ell,Q_T}^{\frac{\ell}{\ell-r}}+L_2\ar^{-1}\|u\|_{2q,Q_T}^{\frac{ q}{q-\lz}} \nonumber\\
 		&&+ \left[c(N,\nu,q)\right]^{\frac{1}{\alpha(r-2j)}}4^{\frac{1}{\alpha^2(r-2j)}}L_1^{-\frac{j}{r-2j}}L_2^{-\frac{2+j\alpha}{(r-2j)\alpha}}\|\vp\|_{\ell,Q_T}^{-\beta_1}\ar^{\frac{j}{ r-2j}}\|u\|_{2q,Q_T}^{-\alpha_1}\|v\|_{2q,Q_T}^{\frac{4}{(r-2j)\alpha}},\label{jt10}
 	\end{eqnarray}
 	where
 	\begin{equation}\label{bo}
 		\beta_1=\frac{j\ell}{(r-2j)(\ell-r)},\ \ \alpha_1=\frac{(j\alpha+2)q}{\alpha(q-\lz)(r-2j)}.
 	\end{equation}
 	In view of \eqref{jm1}, we have
 	\begin{eqnarray*}
 		2\sqrt{N} L_1\|\vp\|_{\ell,Q_T}^{\frac{\ell}{\ell-r}}\leq  2\sqrt{N}L_1\|\vp\|_{\infty,Q_T}.
 	\end{eqnarray*}
 	We take 
 	$$2\sqrt{N} L_1=\frac{1}{2}.$$
 	Plug \eqref{jt10} into \eqref{jm8} and then use the above choice for $L_1$  to obtain
 	\begin{eqnarray*}
 \|\vp\|_{\infty,Q_T}
 		&\leq&4\sqrt{N}L_2\ar^{-1}\|u\|_{2q,Q_T}^{\frac{ q}{q-\lz}}\nonumber\\ &&+4\sqrt{N}\left[c(N,\nu,q)\right]^{\frac{1}{r-2j}}(4\sqrt{N})^{\frac{j}{r-2j}}L_2^{-\frac{2+j\alpha}{(r-2j)\alpha}}\|\vp\|_{\ell,Q_T}^{-\beta_1}\ar^{\frac{j}{ r-2j}}\|u\|_{2q,Q_T}^{-\alpha_1}\|v\|_{2q,Q_T}^{\frac{4}{(r-2j)\alpha}}.
 	\end{eqnarray*}
 	Recall \eqref{pdef} and \eqref{msrd} to derive
 	\begin{eqnarray*}
 		\|u\|_{\infty,Q_T}&\leq& 4\sqrt{N}L_2\|u\|_{2q,Q_T}^{\frac{ q}{q-\lz}}\nonumber\\
 		&&+4\sqrt{N}\left[c(N,\nu,q)\right]^{\frac{1}{r-2j}}(4\sqrt{N})^{\frac{j}{r-2j}}L_2^{-\frac{2+j\alpha}{(r-2j)\alpha}}\|u\|_{\ell,Q_T}^{-\beta_1}\|u\|_{r,Q_T}^{\frac{r-j}{ r-2j}+\beta_1}\|u\|_{2q,Q_T}^{-\alpha_1}\|v\|_{2q,Q_T}^{\frac{4}{(r-2j)\alpha}}.\label{ap5}
 	\end{eqnarray*}
 	As in \eqref{jm1}, we calculate, with the aid of the interpolation inequality,  that
 	\begin{eqnarray*}
 		4\sqrt{N}L_2\|u\|_{2q,Q_T}^{\frac{ q}{q-\lz}}&\leq&4\sqrt{N}L_2\left[\|u\|_{\infty,Q_T}^{\frac{q-\lz}{q}}\|u\|_{2\lz,Q_T}^\frac{\lz}{q}\right]^{\frac{ q}{q-\lz}}\nonumber\\
 		&=&4\sqrt{N}L_2\|u\|_{\infty,Q_T}\|u\|_{2\lz,Q_T}^\frac{\lz}{q-\lz}.
 	\end{eqnarray*}
 	We pick $L_2$ so that the coefficient of $\|u\|_{\infty,Q_T}$ in the last term in the above inequality  is $\frac{1}{2}$, i.e.,
 	\begin{equation*}
 		4\sqrt{N}L_2\|u\|_{2\lz,Q_T}^\frac{\lz}{q-\lz}	=\frac{1}{2}.
 	\end{equation*}
 	Combining this with \eqref{ap5} yields 
 	\begin{eqnarray}
 		\|u\|_{\infty,Q_T}&\leq&8\sqrt{N}c_1\|u\|_{2\lz,Q_T}^{s_1}\|u\|_{\ell,Q_T}^{-\beta_1}\|u\|_{r,Q_T}^{\frac{r-j}{ r-2j}+\beta_1}\|u\|_{2q,Q_T}^{-\alpha_1}\|v\|_{2q,Q_T}^{\frac{4}{(r-2j)\alpha}} ,\label{hhap2}
 	\end{eqnarray}
 	where
 	\begin{eqnarray}%
 		c_1&=&\left[c(N,\nu,q)\right]^{\frac{1}{r-2j}}(4\sqrt{N})^{\frac{j}{r-2j}}(8\sqrt{N})^{\frac{2+j\alpha}{(r-2j)\alpha}}=\left[c(N,\nu,q,j)\right]^{\frac{1}{r-2j}},\label{c1}\\
 		s_1&=&\frac{(2+j\alpha)\lz}{(r-2j)\alpha(q-\lz)}=\frac{c(N,q,j)}{r-2j}.\label{s1}
 	\end{eqnarray}
 	Recall \eqref{uvw} to obtain
 \begin{equation}
 	\|v\|_{2q,Q_T}\leq \|u\|_{2q,Q_T}+\|w\|_{2q,Q_T}.\nonumber
 \end{equation}	
 Fix
 \begin{equation}
 	q_0>N+2.\nonumber
 \end{equation}We may assume
 	\begin{equation}\label{wu}
 		\|w\|_{2q,Q_T}	\leq \|u\|_{2q,Q_T}\ \ \mbox{for each $q\in [q_0, 2q_0]$.}
 	\end{equation}
 	Indeed, if this is not true, i.e., 
 	\begin{equation}
 		\|w\|_{2q,Q_T}	\geq \|u\|_{2q,Q_T}\ \ \mbox{for some $q\in [q_0, 2q_0]$.}\nonumber
 	\end{equation} then we can conclude from \eqref{int}, \eqref{wm}, and \eqref{ns6} that
 	\begin{eqnarray}
 		\|v\|_{2q,Q_T}&\leq& 2\|w\|_{2q,Q_T}\nonumber\\
 		&\leq& 2\|w\|_{\infty,Q_T}^{\frac{q-\lz}{q}}\|w\|_{2\lz,Q_T}^{\frac{\lz}{q}}\nonumber\\
 		&\leq& c(N,\nu)\|\vo\|_{\infty,Q_T}^{\frac{q-\lz}{q}}\|\vo\|_{2,Q_T}^{\frac{\lz}{q}}.\nonumber
 	\end{eqnarray}
 	Combine this with \eqref{pi1} to deduce
 	\begin{eqnarray}
 		\|v\|_{\infty, Q_T}&\leq&2\sqrt{N}\|\vo\|_{\infty,\rn}+c(N,q,\nu)\|\vo\|_{2,\rn}+c(N,\nu,q)\|\vo\|_{\infty,Q_T}^{\frac{2(q-\lz)}{q}}\|\vo\|_{2,Q_T}^{\frac{2\lz}{q}}.\nonumber
 	\end{eqnarray}
 	Note that the constant $c(N,q,\nu)$ is finite for $q\in [q_0, 2q_0]$. Thus, \eqref{ns24} follows. We must point out that the introduction of $q_0$ is necessary due to \eqref{li1} and \eqref{li2}.
 	
 	Under \eqref{wu}, \eqref{hhap2} becomes
 	\begin{eqnarray}
 		\|u\|_{\infty,Q_T}&\leq&8\sqrt{N} 2^{\frac{4}{(r-2j)\alpha}}c_1\|u\|_{2\lz,Q_T}^{s_1}\|u\|_{\ell,Q_T}^{-\beta_1}\|u\|_{r,Q_T}^{\frac{r-j}{ r-2j}+\beta_1}\|u\|_{2q,Q_T}^{\frac{4}{(r-2j)\alpha}-\alpha_1}\nonumber\\
 		&=&8\sqrt{N}c_1\|u\|_{2\lz,Q_T}^{s_1}\|u\|_{\ell,Q_T}^{-\beta_1}\|u\|_{r,Q_T}^{\frac{r-j}{ r-2j}+\beta_1}\|u\|_{2q,Q_T}^{\frac{4}{(r-2j)\alpha}-\alpha_1}.\label{hhap3}%
 	\end{eqnarray}
 	We have incorporated $2^{\frac{4}{(r-2j)\alpha}}$ into $c_1$ due to its definition in \eqref{c1}.
  
 	Next, we will demonstrate how to combine the last three norms in \eqref{hhap3} into a single one. Our first attempt in this direction yields the following result.
 	\begin{clm}\label{clm4}Let $q$ be given as in \eqref{wu}.
 	Define
 	\begin{equation}
 		\mq=\frac{2(q-2\lz)}{\alpha q}.\label{mqd}
 	\end{equation}
 	Then  there exist two positive numbers $c(N,q)$ and $c(N,\nu,q)$ such that
 	\begin{equation}\label{wif}
 		\|u\|_{\infty,Q_T}\leq 8\sqrt{N}\left[c(N,\nu,q)\right]^{\frac{1}{r-\mq}}\|u\|_{2\lz,Q_T}^{\frac{c(N,q)}{r-\mq}}\|u\|_{r,Q_T}^{\frac{r}{r-\mq}}\ \ \mbox{for each $r>\mq$}.
 	\end{equation}
 \end{clm}
 \begin{proof} 
 	Plug \eqref{bo} into \eqref{hhap3} and take $\ell\ra \infty$ in the resulting inequality   to derive
 \begin{eqnarray}
 		\|u\|_{\infty,Q_T}&\leq& 8\sqrt{N}c_1\|u\|_{2\lz,Q_T}^{s_1}\|u\|_{\infty,Q_T}^{-\frac{j}{r-2j}}\|u\|_{r,Q_T}^{\frac{r}{ r-2j}}\|u\|_{2q,Q_T}^{-\frac{(\alpha j+2)q}{\alpha(r-2j)(q-\lz)}+\frac{4}{(r-2j)\alpha}}\nonumber\\
 		&=& 8\sqrt{N}c_1\|u\|_{2\lz,Q_T}^{s_1}\|u\|_{\infty,Q_T}^{-\frac{j}{r-2j}}\|u\|_{r,Q_T}^{\frac{r}{ r-2j}}\|u\|_{2q,Q_T}^{\frac{q(\mq-j)}{(r-2j)(q-\lz)}}.\label{e3}
 	\end{eqnarray}
 		Here $c_1$ and $s_1$ remain the same as before because they do not depend on $\ell$.
 	By virtue of \eqref{int}, we have
 	\begin{eqnarray}
 		\|u\|_{2q,Q_T}^{\frac{jq}{(r-2j)(q-\lz)}}&\leq&\left[	\|u\|_{\infty,Q_T}^{\frac{q-\lz}{q}}	\|u\|_{2\lz,Q_T}^\frac{\lz}{q}\right]^{\frac{jq}{(r-2j)(q-\lz)}}\nonumber\\ &=&\|u\|_{\infty,Q_T}^{\frac{j}{(r-2j)}}\|u\|_{2\lz,Q_T}^{\frac{j\lz}{(r-2j)(q-\lz)}}.\nonumber
 	\end{eqnarray}
 	Incorporating this into \eqref{e3} yields
 	\begin{eqnarray}
 		\|u\|_{\infty,Q_T}&\leq&8\sqrt{N}c_1\|u\|_{2\lz,Q_T}^{s_0}\|u\|_{r,Q_T}^{\frac{r}{ r-2j}}\|u\|_{2q,Q_T}^{\frac{q(\mq-2j)}{(r-2j)(q-\lz)}},\label{e5}
 	\end{eqnarray}
 	where 
 	\begin{equation}
 		s_0=s_1+\frac{j\lz}{(r-2j)(q-\lz)}=\frac{2(1+j\alpha)\lz}{(r-2j)\alpha(q-\lz)}.	\nonumber
 	\end{equation}
 	We choose
 	\begin{equation}
 		j=\frac{\mq}{2}.\nonumber
 	\end{equation}
 	As a result, the last exponent in \eqref{e5} is $0$.
 	Substitute this value of $j$ into \eqref{e5} to  arrive at \eqref{wif}. The proof is complete.
 \end{proof}
 An easy consequence of the preceding claim is that \eqref{intr} holds.
 Indeed, plug \eqref{adef} into \eqref{mqd} to obtain
 \begin{equation}
 	\mq=\frac{(N+2)(q-2\lz)}{q-N-2}.\nonumber
 \end{equation}
 We easily see from \eqref{lzd} that $\mq$ is a strictly decreasing function of $q$ on $(N+2, \infty)$. 
 In addition,
 \begin{equation}\label{mql}
 	\lim_{q\ra\infty}\mq=N+2\ \ \mbox{and}\ \ 	\lim_{q\ra(N+2)^+}\mq=\infty.
 \end{equation} 
 Consequently, we have
 \begin{eqnarray}
 	\mq>N+2=N\lz.\label{mn}
 \end{eqnarray}
 Under the assumption in \eqref{intr}, we  take  $q_0$ to be the unique solution to the equation
 \begin{equation}
 \mq=s.\nonumber
 \end{equation}
 As a result
 \begin{equation}
 	\mq<s\ \ \mbox{for $q\in(q_0, 2q_0]$.}\nonumber
 \end{equation}
 We can pick $q=2q_0, r=s$ in \eqref{wif}. This is a partial recovery of the result by Serrin, Prodi, and Ladyzenskaja we mentioned in the introduction.
 
We easily infer from \eqref{mql} that there is a solution  to the equation
\begin{equation}
	\mq=2q \ \ \mbox{in $(N+2, \infty)$.}\nonumber
\end{equation}
The solution is also unique due to the strict monotonicity of $\mq$.
From here on, we take $q_0$ to be $2$ times the solution 
and further require
\begin{equation}\label{qcon1}
	q\in [q_0, 2q_0].
\end{equation}
Consequently,
\begin{equation}\label{qlb}
\min_{q\in [q_0, 2q_0]}\left(2q-\mq\right)>0.
\end{equation}
Without any loss of generality, we may assume
\begin{equation}\label{fp1}
	f(s)>0 \ \ \mbox{for each $s\in[2q, \infty)$.}
\end{equation}
Indeed, suppose that this is not true. That is,
\begin{equation}
	f(s)\leq0 \ \ \mbox{for some $s\in[2q, \infty)$.}\nonumber
\end{equation}
It immediately follows from \eqref{fd} that 
\begin{equation}
	\|u\|_{s,Q_T}\leq 1.\nonumber
\end{equation}
In view of \eqref{qlb}, we can incorporate this into \eqref{wif} and then apply the first estimate in \eqref{vw} in the resulting inequality  to get
\begin{eqnarray}
	\|u\|_{\infty,Q_T}&\leq &8\sqrt{N}\left[c(N,\nu, q)\right]^{\frac{1}{s-\mq}}\|u\|_{2\lz,Q_T}^{\frac{c(N,q)}{s-\mq}}\nonumber\\
	&\leq&8\sqrt{N} \left[c(N,\nu, q)\right]^{\frac{1}{s-\mq}}\|\vo\|_{2,\rn}^{\frac{c(N,q)}{s-\mq}}\nonumber\\
	&\leq&8\sqrt{N} \left[\max\left\{1,c(N,\nu, q)\right\}\right]^{\frac{1}{2q-\mq}}\left[\max\left\{1, \|\vo\|_{2,\rn}\right\}\right]^{\frac{c(N,q)}{2q-\mq}}.\nonumber
\end{eqnarray}
  With the aid of \eqref{wm}, we calculate
 \begin{eqnarray}
 		\|v\|_{\infty,Q_T}&\leq &	\|u\|_{\infty,Q_T}+	\|w\|_{\infty,Q_T}\nonumber\\
 		&\leq &\sqrt{N}\|\vo\|_{\infty,\rn}+8\sqrt{N}\left[\max\left\{1,c(N,\nu, q)\right\}\right]^{\frac{1}{2q-\mq}}\left[\max\left\{1, \|\vo\|_{2,\rn}\right\}\right]^{\frac{c(N,q)}{2q-\mq}}.\nonumber
 \end{eqnarray}
 We obtain \eqref{ns24} by minimizing the right-hand over the interval in \eqref{qcon1}.


We can also assume that
\begin{equation}
	\mbox{$\frac{\fr}{r}=\ln\left(\|u\|_{r,Q_T}\right)$ is an increasing function of $r$ on $(2q, \infty)$.}\nonumber
\end{equation} 
Obviously, this is a consequence of
\begin{equation}\label{fri1}
	\left(\frac{\fr}{r}\right)^\prime=\frac{r\fp-\fr}{r^2}>0\ \ \mbox{for each $r\in [2q, \infty)$.}
\end{equation}
Suppose that the preceding inequality is not true, i.e., 
\begin{equation}
	\fp\leq \frac{\fr}{r} \ \mbox{for some $r\in [2q, \infty)$.}\nonumber
\end{equation}
This combined with the inequality in \eqref{fpi} implies
\begin{equation}
	\frac{\fr-f(2\lz)}{r-2\lz}\leq  \frac{\fr}{r},\nonumber
\end{equation}
from which it follows
\begin{equation}
	\|u\|_{r, Q_T}\leq\|u\|_{2\lz, Q_T}.\nonumber
\end{equation}
Collect this in \eqref{wif} to obtain
\begin{eqnarray}
	\|u\|_{\infty,Q_T}&\leq& 8\sqrt{N}\left[c(N,\nu, q)\right]^{\frac{1}{r-\mq}}\|u\|_{2\lz,Q_T}^{\frac{c(N,q)}{r-\mq}+\frac{r}{r-\mq}}\nonumber\\
	&\leq &8\sqrt{N}\left[c(N,\nu, q)\right]^{\frac{1}{r-\mq}+\frac{r}{r-\mq}}\|\vo\|_{2,\rn}^{\frac{c(N,q)}{r-\mq}+\frac{r}{r-\mq}}\nonumber\\
	&\leq&8\sqrt{N}\left[\max\left\{1, c(N,\nu, q)\right\}\right]^{\frac{1}{2q-\mq}+\frac{2q}{2q-\mq}}\left[\max\left\{1, \|\vo\|_{2,\rn} \right\}\right]^{\frac{c(N,q)}{2q-\mq}+\frac{2q}{2q-\mq}}.\nonumber
\end{eqnarray}
This implies \eqref{ns24}.




Obviously,  the gap between Claim \ref{clm4} and Theorem \ref{thm} is still substantial. This means that we must make a better choice of our parameters than the one in Claim \ref{clm4}. 
Nonetheless,  Claim \ref{clm4} provides important insights. 
It shows that it is possible to make the last remaining exponent as close to $1$ as possible when we combine the three norms in \eqref{hhap3}. Unfortunately, in the preceding Claim the parameter $\mq$ is too large. To see why this is the problem, we make use of \eqref{int} in \eqref{wif} to obtain
\begin{equation}
	\|u\|_{\infty,Q_T}\leq 8\sqrt{N}\left[c(N,\nu, q)\right]^{\frac{1}{r-\mq}} \|v\|_{2\lz,Q_T}^{\frac{c(N,q)}{r-\mq}+\frac{2\lz}{r-\mq}}\|u\|_{\infty,Q_T}^{\frac{r-2\lz}{r-m_q}}.\nonumber
\end{equation}
This inequality is useful only when
\begin{equation}
	\frac{r-2\lz}{r-M_q}<1.\nonumber
\end{equation}
This necessitates $2\lz>\mq$, which contradicts \eqref{mn}.

On the other hand, we recall \eqref{bo} to calculate the total sum of the last three exponents in \eqref{hhap2}, thereby obtaining
\begin{eqnarray}\label{tot}
	\lefteqn{\frac{r-j}{ r-2j}+\frac{4}{\alpha(r-2j)}-\alpha_1}\nonumber\\
	&	=&\frac{r-j}{ r-2j}+\frac{4}{\alpha(r-2j)}-\frac{(j\alpha+2)q}{\alpha(q-\lz)(r-2j)}\nonumber\\
	&=&1+\frac{	2(q-2\lz)-\lz\alpha j}{\alpha(r-2j)(q-\lz)}.
\end{eqnarray}
 Thus, the total sum can be made less that $1$, provided that we choose
\begin{eqnarray}
	j&\geq&\frac{2(q-2\lz)}{\lz\alpha}
	.\label{jl1}
\end{eqnarray} 
	This suggests that our goal is achievable as long as we can combine norms in such a way that the total sum of the exponents involved does not increase too much. 	
As we noted before, whenever we move a positive exponent from a small indexed norm to a larger one via the interpolation inequality it gets smaller. The reverse, which is more or less  the third case mentioned in the introduction, is much more tricky. This constitutes the main difficulty. The opposite is true with  negative exponents.
	

Our goal is to show
	\begin{equation}
		\|u\|_{\infty, Q_T}\leq c_1\|u\|_{2\lz,Q_T}^{s_1}\|u\|_{r, Q_T}^\theta\ \ \mbox{for some $\theta\in\left(0, \frac{r}{r-2\lz}\right)$},\nonumber
	\end{equation}
	and then apply \eqref{int}. 
	We will proceed along this line of thinking.  It turns out that relations among various parameters are rather delicate.
	


	To continue the proof of \eqref{ns24},  
	plug \eqref{bo} into \eqref{hhap3} to deduce
	\begin{eqnarray}
		\|u\|_{\infty,Q_T}&\leq&8\sqrt{N}\  c_1\|u\|_{2\lz,Q_T}^{s_1}\|u\|_{\ell,Q_T}^{-\frac{j\ell}{(r-2j)(\ell-r)}}\|u\|_{r,Q_T}^{\frac{r-j}{ r-2j}+\frac{j\ell}{(r-2j)(\ell-r)}}\|u\|_{2q,Q_T}^{\frac{q(\mq-j)}{(r-2j)(q-\lz)}}\nonumber\\
		&=&8\sqrt{N}\  c_1\|u\|_{2\lz,Q_T}^{s_1}\|u\|_{\ell,Q_T}^{-\frac{j\ell}{(r-2j)(\ell-r)}}\|u\|_{r,Q_T}^{\frac{r(\ell-r+j)}{(r-2j)(\ell-r)}}\|u\|_{2q,Q_T}^{\frac{q(\mq-j)}{(r-2j)(q-\lz)}}.\label{hh1}
	\end{eqnarray}
	We first eliminate  the $L^{2q}$- norm. For this purpose,
	we need to require
	\begin{equation}\label{r2q}
		r>2q.
	\end{equation}
	This together with \eqref{ldef} enables us to form the interpolation inequality
	\begin{eqnarray}
		\|u\|_{r,Q_T}\leq \|u\|_{\ell,Q_T}^{\frac{\ell(r-2q)}{r(\ell-2q)}}\|u\|_{2q,Q_T}^{\frac{2q(\ell-r)}{r(\ell-2q)}}.\label{int2}
	\end{eqnarray}
	Pick
	\begin{equation}\label{jmq}
		j>\mq.
	\end{equation}
Subsequently, we can raise both sides of \eqref{int2} to the power of $\frac{(j-\mq)r(\ell-2q)}{2(\ell-r)(r-2j)(q-\lz)}$ to get
	\begin{equation}
		\|u\|_{r,Q_T}^{\frac{(j-\mq)r(\ell-2q)}{2(\ell-r)(r-2j)(q-\lz)}}\leq \|u\|_{\ell,Q_T}^{\frac{(j-\mq)\ell(r-2q)}{2(\ell-r)(r-2j)(q-\lz)}}\|u\|_{2q,Q_T}^{\frac{q(j-\mq)}{(r-2j)(q-\lz)}}\nonumber
	\end{equation}
	Incorporate this into \eqref{hh1} and keep \eqref{fd}, \eqref{c1}, and \eqref{s1} in mind to deduce
	\begin{eqnarray}
		\lefteqn{\|u\|_{\infty,Q_T}}\nonumber\\
		&\leq&8\sqrt{N} c_1\|u\|_{2\lz,Q_T}^{s_1}\|u\|_{\ell,Q_T}^{-\frac{j\ell}{(r-2j)(\ell-r)}+\frac{(j-\mq)\ell(r-2q)}{2(\ell-r)(r-2j)(q-\lz)}}\|u\|_{r,Q_T}^{\frac{r(\ell-r+j)}{(r-2j)(\ell-r)}-\frac{(j-\mq)r(\ell-2q)}{2(\ell-r)(r-2j)(q-\lz)}}\nonumber\\
		&=&8\sqrt{N}\ c_1\|u\|_{2\lz,Q_T}^{s_1}e^{\frac{f(\ell)}{(r-2j)(\ell-r)}\left[-j+\frac{(j-\mq)(r-2q)}{2(q-\lz)}\right]+\frac{f(r)}{(r-2j)(\ell-r)}\left[\ell-r+j-\frac{(j-\mq)(\ell-2q)}{2(q-\lz)}\right]}\nonumber\\
			&=&8\sqrt{N}\ \left[c(N,\nu, q, j)\right]^{\frac{1}{r-2j}}\|u\|_{2\lz,Q_T}^{\frac{c(N, q, j)}{r-2j}}e^{\frac{f(\ell)-f(r)}{(r-2j)(\ell-r)}\left[-j+\frac{(j-\mq)(r-2q)}{2(q-\lz)}\right]+\frac{f(r)}{r-2j}\left[1-\frac{j-\mq}{2(q-\lz)}\right]}\nonumber\\
		&\ra&8\sqrt{N}\ \left[c(N,\nu, q, j)\right]^{\frac{1}{r-2j}}\|u\|_{2\lz,Q_T}^{\frac{c(N, q, j)}{r-2j}}e^{\frac{f^\prime(r)}{(r-2j)}\left[-j+\frac{(j-\mq)(r-2q)}{2(q-\lz)}\right]+\frac{f(r)}{r-2j}\left[1-\frac{j-\mq}{2(q-\lz)}\right]}\ \ \ \mbox{as $\ell\ra r^+$}.
		\label{ha1}
	\end{eqnarray}
	To continue, we need to further refine our choice of parameters. We are tempted to assume 
	\begin{equation}
		-j+\frac{(j-\mq)(r-2q)}{2(q-\lz)}\leq0.\nonumber
	\end{equation}
	This combined with \eqref{rcon} implies
	\begin{equation}
			2j<2q+\frac{2(q-\lz)j}{j-\mq}.\nonumber
	\end{equation}
	Unfortunately, this  restriction on how large $j$ can be  is inconsistent with \eqref{jl1}, and it does lead to failure. Therefore, we must
	select $j$ so  that
	\begin{equation}\label{ji}
		2j>2q+\frac{2(q-\lz)j}{j-\mq}.
	\end{equation}
	Under this selection, 
	 \eqref{rcon} implies both \eqref{r2q} and
	\begin{equation}\label{r0p}
		-j+\frac{(j-\mq)(r-2q)}{2(q-\lz)}>0.
	\end{equation}
Thus, we are essentially facing	 the third case  mentioned in the introduction. 

	We may write \eqref{ji} as
	\begin{equation}
		(j-q)^2-(\mq-\lz)(j-q)-q(q-\lz)>0.\nonumber
	\end{equation}
	The left-hand side is a quadratic function in $j-q$. We can solve the preceding inequality via the quadratic formula. Upon doing so, we  deduce
	\begin{eqnarray}
		j&>&q+\frac{\mq-\lz+\sqrt{(\mq-\lz)^2+4q(q-\lz)}}{2}\nonumber\\
		&=&q+\frac{\mq-\lz+\sqrt{(2q-\lz)^2+\mq(\mq-2\lz)}}{2}.\label{jl6}
	\end{eqnarray}
	This condition is  stronger  than \eqref{jmq}. To see this, we calculate from \eqref{mn} and \eqref{qlb} that
	\begin{eqnarray}
		\lefteqn{q+\frac{\mq-\lz+\sqrt{(2q-\lz)^2+\mq(\mq-2\lz)}}{2}}\nonumber\\
		&>&2q+\frac{\mq-2\lz}{2}>\mq.\nonumber
	\end{eqnarray}
	
	For later purpose, we need to eliminate $8\sqrt{N}$ from the right-hand side of \eqref{ha1}. To do this, we fix
	\begin{equation}
		\eta>0.\nonumber
	\end{equation}
	We may claim
	\begin{equation}\label{ri1}
		8\sqrt{N}\|u\|_{r,Q_T}^{-\frac{\eta
				r}{2(q-\lz)(r-2j)}}\leq 1 \ \ \mbox{for each $r\in (2j, \infty)$ and each $q\in [q_0,2q_0]$ . }
	\end{equation}
	If this is not true, then
	\begin{equation}
		\|u\|_{r,Q_T}\leq \left(	8\sqrt{N}\right)^{\frac{2(q-\lz)(r-2j)}{\eta r}}\ \ \mbox{for some $r\in (2j, \infty)$ and some $q\in [q_0,2q_0]$ . }\nonumber
	\end{equation}
	This combined with \eqref{wif} implies
	\begin{eqnarray}
		\|u\|_{\infty,Q_T}&\leq& \left(8\sqrt{N}\right)^{1+\frac{2(q-\lz)(r-2j)}{\eta(r-\mq)}}\left[c(N,\nu,q)\right]^{\frac{1}{r-\mq}}\|u\|_{2\lz,Q_T}^{\frac{c(N,q)}{r-\mq}}\nonumber\\
		&\leq&\left(8\sqrt{N}\right)^{1+\frac{2(q-\lz)}{\eta}}\left[\max\left\{1,c(N,\nu,q)\right\}\right]^{\frac{1}{2q-\mq}}\left(\max\left\{1,\|\vo\|_{2,\rn}\right\}\right)^{\frac{c(N,q)}{2q-\mq}}.\nonumber
	\end{eqnarray}
	The right -hand side is finite due to \eqref{qlb}, and \eqref{ns24} follows.
	
	Under \eqref{ri1}, \eqref{ha1} becomes
	\begin{equation}\label{ha2}
	\|u\|_{\infty,Q_T}\leq \left[c(N,\nu, q, j)\right]^{\frac{1}{r-2j}}\|u\|_{2\lz,Q_T}^{\frac{c(N, q, j)}{r-2j}}e^{\frac{f^\prime(r)}{(r-2j)}\left[-j+\frac{(j-\mq)(r-2q)}{2(q-\lz)}\right]+\frac{f(r)}{r-2j}\left[1-\frac{j-\mq-\eta}{2(q-\lz)}\right]}.
	\end{equation}
	 For each
	\begin{equation}
	\ve\in \left(0, 1\right),\nonumber
	\end{equation}
we consider the function
	\begin{equation}
		h(s)=\frac{f(s)}{s^{1-\ve}(s-2j)^\ve} \ \ \mbox{on $(2j, \infty)$.}\nonumber
	\end{equation}
Recall \eqref{fp1} to obtain
	\begin{equation}\label{2jl}
		\lim_{s\ra (2j)^+}h(s)=+\infty.
	\end{equation}
	On the other hand,  \eqref{ffl} asserts 
	\begin{equation}\label{3jl}
			\lim_{s\ra\infty}h(s)=\lim_{s\ra\infty}\frac{f(s)}{s}\left(\frac{s}{s-2j}\right)^\ve=\ln\left(\|u\|_{\infty, Q_T}\right).
	\end{equation}
	We may assume that 
	\begin{equation}
		\|u\|_{\infty, Q_T}>1.\nonumber
	\end{equation}
	Were this not true, we would have nothing more to prove.
	The Intermediate Value Theorem asserts that the equation
	\begin{equation}\label{eq}
		h(s)=(1+\ve)\ln\left(\|u\|_{\infty, Q_T}\right) 
	\end{equation} has at leas one solution in the interval $(2j, \infty)$. Obviously, the largest solution exists, and we denote it by $\re$. 
	We claim 
	\begin{equation}\label{rel}
		\re\ra\infty\ \ \mbox{as $\ve\ra 0^+$}\ \ \mbox{	and}\ \ 	h^\prime(\re)=\frac{f^\prime(\re)-f(\re)\left(\frac{1-\ve}{\re}+\frac{\ve}{\re-2j}\right)}{\re^{1-\ve}(\re-2j)^\ve}\leq 0.
	\end{equation}
To prove the inequality here, we argue by contradiction. 	Suppose that it is not  true, i.e.,
	\begin{equation}
		h^\prime(\re)> 0.\nonumber
	\end{equation}
It immediately follows that $h(s)$ is increasing in a neighborhood of $\re$. As a result, we have
	\begin{equation}
		h(s)>(1+\ve)\ln\left(\|u\|_{\infty, Q_T}\right)\ \ \mbox{for some $s\in (\re, \infty)$.}\nonumber
	\end{equation}
	This together with \eqref{3jl} puts us in a position to apply the Intermediate Value Theorem again. Upon doing so, we conclude  that 
	there is a solution to \eqref{eq} in the interval $ (\re, \infty)$, which is a contradiction because $\re$ is the largest solution. This completes the proof.

To see the limit in \eqref{rel},  we set
	\begin{equation}
		r_0=\liminf_{\ve\ra 0^+}\re.\nonumber
	\end{equation}
	Remember
	\begin{equation}
		h(\re)=(1+\ve)\ln\left(\|u\|_{\infty, Q_T}\right).\nonumber
	\end{equation}
This combined with \eqref{2jl} implies that 
\begin{equation}
	r_0>2j.\nonumber
\end{equation}	
 Note that 
the inequality in \eqref{rel}  is equivalent to
\begin{equation}\label{seq}
	f^\prime(\re)\leq f(\re)\left(\frac{1-\ve}{\re}+\frac{\ve}{\re-2j}\right).
\end{equation}
If $r_0$ is finite, we can take $\ve\ra 0^+$ in the preceding inequality to derive
\begin{equation}
	f^\prime(r_0)\leq \frac{f(r_0)}{r_0}.\nonumber
\end{equation}
But this contradicts \eqref{fri1}. Thus, we must have
\begin{equation}
	r_0=\infty.\nonumber
\end{equation}
 Obviously, \eqref{ha2} holds for $r=\re$. That is, there holds
	\begin{equation}
		\|u\|_{\infty,Q_T}\leq 	\left[c(N,\nu, q, j)\right]^{\frac{1}{\re-2j}}\|u\|_{2\lz,Q_T}^{\frac{c(N, q, j)}{\re-2j}}e^{\frac{f^\prime(\re)}{(\re-2j)}\left[-j+\frac{(j-\mq)(\re-2q)}{2(q-\lz)}\right]+\frac{f(\re)}{\re-2j}\left[1-\frac{j-\mq-\eta}{2(q-\lz)}\right]}.\nonumber
	\end{equation}
	In view of \eqref{r0p}, we can incorporate \eqref{seq} into this to derive
	\begin{eqnarray}
		\|u\|_{\infty,Q_T}&\leq	&\left[c(N,\nu, q, j)\right]^{\frac{1}{\re-2j}}\|u\|_{2\lz,Q_T}^{\frac{c(N, q, j)}{\re-2j}}e^{\frac{f(\re)\left(\frac{1-\ve}{\re}+\frac{\ve}{\re-2j}\right)}{(\re-2j)}\left[-j+\frac{(j-\mq)(\re-2q)}{2(q-\lz)}\right]+\frac{f(\re)}{\re-2j}\left[1-\frac{j-\mq-\eta}{2(q-\lz)}\right]}\nonumber\\
		&=&\left[c(N,\nu, q, j)\right]^{\frac{1}{\re-2j}}\|u\|_{2\lz,Q_T}^{\frac{c(N, q, j)}{\re-2j}}\|u\|_{\re,Q_T}^{\beta_\ve},\label{ha6}
	\end{eqnarray}
	where
	\begin{eqnarray}
\beta_\ve&=&\frac{\re}{\re-2j}\left[\left[\frac{1}{\re}+\frac{2j\ve}{\re(\re-2j)}\right]\left[-j+\frac{(j-\mq)(\re-2q)}{2(q-\lz)}\right]+1-\frac{j-\mq-\eta}{2(q-\lz)}\right]\nonumber\\
		&=&\frac{2j\ve}{(\re-2j)^2}\left[-j+\frac{(j-\mq)(\re-2q)}{2(q-\lz)}\right]+\frac{1}{\re-2j}\left[\re-j-\frac{2q(j-\mq)-\eta}{2(q-\lz)}\right].\nonumber
	\end{eqnarray}
	It is easy to see that
	\begin{equation}
		\lim_{\ve\ra 0^+}\beta_\ve=1.\nonumber
	\end{equation}
	According to \eqref{r0p}, the first term in the expression for $\beta_\ve$ is positive. Therefore,
	\begin{equation}
		\beta_\ve>0\ \ \mbox{whenever $\re>j+\frac{2q(j-\mq)-\eta}{2(q-\lz)}$}.\nonumber
	\end{equation}
	In view of the limit in \eqref{rel}, we may assume that the last inequality holds for all $\ve\in(0,1)$.
We are in a position to apply \eqref{int} in \eqref{ha6}. Upon doing so, we arrive at 
	\begin{equation}\label{ha7}
		\|u\|_{\infty,Q_T}^{\re\left[1-\frac{(\re-2\lz)\beta_\ve}{\re}\right]}\leq\left[c(N,\nu, q, j)\right]^{\frac{\re}{\re-2j}}\|u\|_{2\lz,Q_T}^{\frac{c(N, q, j)\re}{\re-2j}+2\lz\beta_\ve}.
	\end{equation}
	Note that
	\begin{eqnarray}
		\re\left[1-\frac{(\re-2\lz)\beta_\ve}{\re}\right]&=&	\re-(\re-2\lz)\beta_\ve\nonumber\\
		&=&	\re-\frac{2j\ve(\re-2\lz)}{(\re-2j)^2}\left[-j+\frac{(j-\mq)(\re-2q)}{2(q-\lz)}\right]\nonumber\\
		&&-\frac{(\re-2\lz)}{\re-2j}\left[\re-j-\frac{2q(j-\mq)-\eta}{2(q-\lz)}\right]\nonumber\\
		&=&j+\frac{2q(j-\mq)-\eta}{2(q-\lz)}-\frac{2j\ve(\re-2\lz)}{(\re-2j)^2}\left[-j+\frac{(j-\mq)(\re-2q)}{2(q-\lz)}\right]\nonumber\\
		&&-\frac{2(j-\lz)}{\re-2j}\left[\re-j-\frac{2q(j-\mq)-\eta}{2(q-\lz)}\right]\nonumber\\
		&\ra&j+\frac{2q(j-\mq)-\eta}{2(q-\lz)}-2(j-\lz)\ \ \mbox{ as $\ve\ra0^+$.}\nonumber
	\end{eqnarray}
	Pass to the limit in \eqref{ha7} to obtain
		\begin{equation}
		\|u\|_{\infty,Q_T}^{\frac{\lz j}{q-\lz}+2\lz-\frac{2q\mq+\eta}{2(q-\lz)}}\leq c(N,\nu, q, j)\|u\|_{2\lz,Q_T}^{c(N, q, j)+2\lz}.\nonumber
	\end{equation}
	We further require the exponent on the left-hand side to be positive, i.e., 
	\begin{equation}\label{jl7}
		j>\frac{(\mq-2\lz)q}{\lz}+2\lz+\frac{\eta}{2\lz}.
	\end{equation}
Consequently,
\begin{equation}
		\|u\|_{\infty,Q_T}\leq c(N,\nu, q, j,\eta)\|u\|_{2\lz,Q_T}^{c(N, q, j,\eta)}\leq c(N,\nu, q, j,\eta)\|\vo\|_{2,\rn}^{c(N, q, j,\eta)}.\nonumber
\end{equation}
	Obviously, this implies \eqref{ns24}.

In summary, the order in which we pick our parameters is as follows: 
We first choose $q$ as in \eqref{qcon1}, which implies \eqref{qcon}. Take any $\eta>0$. Then select $j$ so that  \eqref{jl6} and \eqref{jl7} are both satisfied.

		Remember that all the constants in our estimates do not depend $T$. This enables us to extend the local-in-time solution as a global one. The proof of Theorem \ref{thm} is complete.
	\end{proof}
	
	Naturally, one would ask if our choice of parameters is the best possible. Is there a better to select parameters? 
	We will investigate this possibility in a future study.
\bigskip


\end{document}